\documentclass[a4paper,12pt]{amsart}
\usepackage[T1]{fontenc}    
\usepackage[utf8]{inputenc}
\usepackage{enumerate}
\usepackage{amssymb}
\usepackage{fullpage}
\usepackage{color}
\usepackage{mathrsfs}
\def\RR{\mathbb R}

\def\NN{\mathbb N}
\def\uu{\mathcal U}
\def\vv{\mathcal V}
\def\ss{\mathcal S}
\def\hh{\mathcal H}
\def\dd{\mathcal D}
\newcommand{\set}[1]{\left\lbrace #1\right\rbrace}
\providecommand{\abs}[1]{\left\lvert#1\right\rvert}

\newcommand{\remove}[1]{ }
\newcommand{\qtq}[1]{\quad\text{#1}\quad}

\newtheorem{theorem}{Theorem}[section]
\newtheorem{lemma}[theorem]{Lemma}

\newtheorem*{t11}{Theorem 1.1}
\newtheorem*{t12}{Theorem 1.2}
\newtheorem*{t14}{Theorem 1.4}
\newtheorem*{t15}{Theorem 1.5}
\newtheorem*{t16}{Theorem 1.6}

\theoremstyle{definition}
\newtheorem{definition}[theorem]{Definition}

\theoremstyle{remark}
\newtheorem{remark}[theorem]{Remark}
\remove{}
\newtheorem{example}[theorem]{Example}
\newtheorem{examples}[theorem]{Examples}

\numberwithin{equation}{section}

\begin{document}
\title[]{Difference of Cantor sets and frequencies in  Thue--Morse type sequences}
\author{Yi Cai*}
\address{School of Mathematical Sciences\\
East China Normal University\\
Shanghai 200241\\
People’s Republic of China}
\email{52170601013@stu.ecnu.edu.cn}

\author{Vilmos Komornik}
\address{College of Mathematics and Computational Science\\
Shenzhen University,
Shenzhen 518060\\
People’s Republic of China,
and
Département de mathématique\\
Université de Strasbourg,
7 rue René Descartes
67084 Strasbourg Cedex\\
France}
\email{vilmos.komornik@math.unistra.fr}

\subjclass[2010]{Primary 28A80, Secondary 11K55, 37B10}
\keywords{Intersection of Cantor sets, Hausdorff dimension,  unique $\beta$-expansion, self-similar sets, Thue--Morse sequence, frequency of words.}
\date{Submitted on December 22, 2019}
\thanks{*Corresponding author; email address: 52170601013@stu.ecnu.edu.cn.}
\thanks{This work has been done during the first authors's stay in the
Department of Mathematics of the University of Strasbourg and was supported by NSFC No. 11671147 and Science and Technology Commission of Shanghai Municipality (STCSM) No. 18dz2271000.
He thanks the members of the department for their hospitality.
The work was of the second author was partially supported by the grant No. 11871348 of the National Natural Science Foundation of China (NSFC)}

\begin{abstract}
In a recent paper, Baker and Kong have studied the Hausdorff dimension of the intersection  of Cantor sets with their translations.
We extend their results to  more general Cantor sets.
The proofs rely on the frequencies of digits in unique expansions in non-integer bases.
In relation with this, we introduce a practical method to determine the frequency of any given finite block in Thue--Morse type sequences.
\end{abstract}
\maketitle

\section{Introduction}\label{s1}

Given a real number $q>1$ and a finite set $\Omega$ of real numbers, a sequence $(c_i)_{i=1}^{\infty} \in \Omega^{\NN}$ is called an \emph{expansion} of $x$ in \emph{base} $q$ over the \emph{alphabet} $\Omega$ if
\begin{equation*}
\sum_{i=1}^{\infty}\frac{c_i}{q^i}=x.
\end{equation*}
If $\Omega_m=\set{0,1,\ldots,m}$ for some positive integer $m$, then
\begin{equation}\label{11}
\Gamma_{q,m}:=\set{\sum_{i=1}^{\infty}\frac{c_i}{q^i}\ :\ c_i\in\set{0,1,\ldots,m}}
\end{equation}
is equal to the interval $[0,\frac{m}{q-1}]$ if $q\le m+1$, and it is a Cantor set in this interval if $q>m+1$ (see \cite{KR}).

In their recent paper \cite{BD}, Baker and Kong studied the size of the intersection sets $\Gamma_{q,1}\cap(\Gamma_{q,1}+t)$ where $t$ is a given real number.
The purpose of this paper is to extend their results to the more general case of the sets $\Gamma_{q,m_1}\cap(\Gamma_{q,m_2}+t)$, where $m_1$ and $m_2$ are arbitrary positive integers.
The set $\Gamma_{q,m_1}\cap(\Gamma_{q,m_2}+t)$ is nonempty if and only if $t\in \Gamma_{q,m_1}-\Gamma_{q,m_2}$, or equivalently, if $t$ has an expansion in base $q$ over the alphabet
\begin{equation}\label{12}
\Omega_{-m_2,m_1}:=\set{-m_2,\ldots,-1,0,1,\ldots,m_1}.
\end{equation}
Since
\begin{equation}\label{13}
\sum_{i=1}^{\infty}\frac{t_i}{q^i}=t
\Longleftrightarrow
\sum_{i=1}^{\infty}\frac{t_i+m_2}{q^i}=t+\frac{m_2}{q-1},
\end{equation}
the expansions over $\Omega_{-m_2,m_1}$ are closely related to the expansions over the more familiar alphabet $\Omega_{m_1+m_2}:=\set{0,\ldots,m_1+m_2}$.

We will consider the above question in the case where $t$ has a unique expansion in base $q$ over the alphabet $\Omega_{-m_2,m_1}$.
We let $\uu_{m_1,m_2}(q)$  denote the set of these real numbers $t$, and
$\uu'_{m_1,m_2}(q)\subset\set{-m_2,\ldots,m_1}^{\NN}$
denote the set of corresponding expansions $(t_i)$.
Furthermore, we set
\begin{equation}\label{14}
\dd_{m_1,m_2}(q):=\set{\dim_H(\Gamma_{q,m_1} \cap (\Gamma_{q,m_2}+t)): t \in \uu_{m_1,m_2}(q)}
\end{equation}
for brevity.

We will prove the following theorem.
We assume without loss of generality that $m_1\ge m_2$.

\begin{theorem}\label{t11}
Let $q>1$.
Then:
\begin{enumerate}[\upshape (i)]
\item If $q-1\le m_2\le m_1$, then
\begin{equation*}
\dd_{m_1,m_2}(q)=\set{0,1}.
\end{equation*}
\item If $m_2<q-1\le m_1$, then
\begin{equation*}
\dd_{m_1,m_2}(q)=\set{0,\frac{\log (m_2+1)}{\log q}}.
\end{equation*}
\item If $m_2<m_1<q-1$, then
$\dd_{m_1,m_2}(q)$ contains the point $\frac{\log (m_2+1)}{\log q}$.\end{enumerate}
\end{theorem}

In order to state our next theorem we recall some results on unique expansions in non-integer bases (see Section \ref{s3} for more details).
We let $\uu$ denote the set of bases in which $1$ has a unique expansion, and $\vv$ denote the set of bases in which $1$ has a unique doubly infinite expansion over the alphabet $\set{0,1,\ldots,m_1+m_2}$.
Note that $\vv$ is a closed set, and $\uu\subsetneq\overline{\uu}\subsetneq\vv$ where $\overline{\uu}$ denotes the topological closure of $\uu$; see \cite{KL-2007,K-2012,VKL}.
Furthermore, $\uu$ has a smallest element $q_{KL}$ known  as the \emph{Komornik--Loreti constant}; see \cite{KL,AC}.
The set $\vv\cap(1,q_{KL})$ is formed by an increasing sequence $q_1<q_2<\cdots$, where $q_1$ is the smallest element of $\vv$: the so called \emph{generalized Golden Ratio}, and $q_k$ converges to $q_{KL}$ as $k\to\infty$; see \cite{B,VKL}.
We also recall from \cite{VKL} that $1$ has a finite greedy expansion of the form $\omega_k0^{\infty}$ in base $q_k$ with a suitable finite word  $\omega_k$ having a nonzero last digit; see Section \ref{s3} for their explicit expression.

We emphasize that the bases $q_{KL}$ and $q_k$ depend on the alphabet $\set{0,1,\ldots,m_1+m_2}$.

\begin{theorem}\label{t12}
Assume that  $m_1=m_2=m$.
Then:
\begin{enumerate}[\upshape (i)]
\item If $q_{KL}<q<\infty$, then $\dd_{m,m}(q)$ contains an interval.
\item If $q=q_{KL}$,  then $\dd_{m,m}(q)$ is formed by the numbers
\begin{align*}
&0,\quad \frac{\log (m+1)}{\log q},
\quad \frac{\log [m^2(m+1)]}{3\log q},\intertext{and}
&\frac{\log m}{\log q}-\frac{\log \frac{m+1}{m}}{\log q}\sum_{i=1}^j\left(-\frac{1}{2}\right)^i
\qtq{for}1\le j<\infty.
\end{align*}
\item If  $q_k<q\le q_{k+1}$ for some $k\ge 1$, then $\dd_{m,m}(q)$ is formed by the numbers
\begin{equation*}
0,\quad \frac{\log (m+1)}{\log q},\qtq{and}
\frac{\log m}{\log q}
-\frac{\log \frac{m+1}{m}}{\log q}\sum_{i=1}^j\left(-\frac{1}{2}\right)^i
\qtq{for $1\le j<k$ if $k\ge 2$.}
\end{equation*}
More precisely, if $(m+t_i)$ ends with with $(\omega_j\overline{\omega_j})^{\infty}$ for some $j\ge1$, then
\begin{equation*}
\dim_H(\Gamma_{q,m} \cap (\Gamma_{q,m}+t))
=\frac{\log m}{\log q}
-\frac{\log \frac{m+1}{m}}{\log q}\sum_{i=1}^j\left(-\frac{1}{2}\right)^i.
\end{equation*}
\end{enumerate}
\end{theorem}

\begin{remark}\label{r13}
Since $m_1=m_2=m$, we have $q_{KL}>q_1=m+1$, so that $q-1>m$ under the assumptions of  Theorem \ref{t12}.
\end{remark}

For the proof of Theorem  \ref{t12} we will need a result (see Lemma \ref{l61} (ii) below) about the determination of the frequency of a given word in the Thue--Morse sequence
\begin{equation*}
(\tau_i)_{i=0}^{\infty}=0110\ 1001\ 1001\ 0110\ \cdots .
\end{equation*}
We recall that the sequence $(\tau_i)$ is defined recursively by $\tau_0:=0$ and
\begin{equation*}
\tau_{2^n}\cdots\tau_{2^{n+1}-1}:=\overline{\tau_0\cdots\tau_{2^n-1}}\qtq{for}n=0,1,\ldots .
\end{equation*}

It follows  from the definition that $\tau_{2i}\tau_{2i+1}\in\set{01,10}$ for every nonnegative integer $i$, and therefore the frequencies of the digits $0$ and $1$ are equal:
\begin{equation*}
\frac{\tau_0+\cdots+\tau_n}{n+1}\to \frac{1}{2}\qtq{as}n\to\infty.
\end{equation*}
We will prove a general theorem on the determination of the  frequencies of arbitrary finite blocks in Thue--Morse type sequences; as a very special case, it provides a new proof of Lemma \ref{l61} (ii).
Instead of the Thue--Morse sequence we  consider more general \emph{mirror sequences} \cite{A}.
Fix a positive integer $M$ and a finite word $\tau_0\cdots\tau_{\ell-1}\in\set{0,\ldots,M}^{\ell}$ of length $\ell\ge 1$.
Defining the reflection of a word $w=c_1\cdots c_n$ by the formula
\begin{equation*}
\overline{w}=\overline{c_1\cdots c_n}
:=(M-c_1)\cdots(M-c_n),
\end{equation*}
we extend $\tau_0\cdots\tau_{\ell-1}$ into an infinite sequence $(\tau_i)_{i=0}^{\infty}$ by the recursive formula
\begin{equation*}
\tau_{2^n\ell}\cdots\tau_{2^{n+1}\ell-1}:=\overline{\tau_0\cdots\tau_{2^n\ell-1}},\quad n=0,1,\ldots .
\end{equation*}
For example, the unique expansions of the so-called \emph{de Vries--Komornik constants} have this form \cite[p. 187]{DL}.

Given two words $\delta$ and $\epsilon$, let $\abs{\delta}$ denote the length of $\delta$, and let $N^\delta(\epsilon)$ denote the number of occurrences of $\delta$ in $\epsilon$, where $\abs{\delta}\le \abs{\epsilon}$.
The following theorem has an independent interest:

\begin{theorem}\label{t14}
Every finite word $\delta$ in $(\tau_i)_{i=0}^{\infty}$ has a density
\begin{equation*}
d(\delta):=\lim_{m\to\infty}\frac{N^{\delta}(\tau_0\cdots\tau_{m-1})}{m}.
\end{equation*}
Moreover, for every integer $n$ satisfying $4^n\ell\ge \abs{\delta}-1$ we have
\begin{equation*}
d(\delta)=d(\overline{\delta})=\frac{P-N}{6\cdot 4^n\ell}
\end{equation*}
with
\begin{align*}
N&=N^\delta(\tau_0\cdots\tau_{4^n\ell-1})
+N^{\overline{\delta}}(\tau_0\cdots\tau_{4^n\ell-1})\intertext{and}
P&=N^\delta(\tau_0\cdots\tau_{4^{n+1}\ell-1})
+N^{\overline{\delta}}(\tau_0\cdots\tau_{4^{n+1}\ell-1}).
\end{align*}
\end{theorem}

In our next two results we use the critical number
\begin{equation}\label{15}
q_{c}=\frac{m+2+\sqrt{m(m+4)}}{2},
\end{equation}
introduced in \cite{DLD}, which is the positive solution of the equation
\begin{equation}\label{16}
\frac{m+1}{q_c}+\sum_{i=2}^{\infty}\frac{m}{q_c^i}=1
\end{equation}
for the alphabet $\Omega_{2m}=\set{0,\ldots,2m}$.

\begin{theorem}\label{t15}\mbox{}
Assume that  $m_1=m_2=m$, and set
\begin{equation*}
c_1:=\frac{\log m}{\log q},\quad
c_2:=\frac{\log (m+1)}{\log q}.
\end{equation*}

\begin{enumerate}[\upshape (i)]
\item If $q\in [q_c,\infty)$, then $\dd_{m,m}(q)$ contains the interval $[c_1,c_2]$.
\item If $q\in (m+1,q_c)$, then there exists a $\delta>0$ such that
\begin{equation*}
\dd_{m,m}(q)\cap(c_2-\delta,c_2)=\varnothing.
\end{equation*}
\end{enumerate}
\end{theorem}

In order to formulate our last theorem we introduce the sets
\begin{align*}
\ss_{m,m}(q)
&:=\set{t \in \uu_{m,m}(q): \Gamma_{q,m} \cap (\Gamma_{q,m}+t) \; \textrm{is a self-similar set}}
\intertext{and}
\hh_{m,m}(q)&:=\set{\dim_H(\Gamma_{q,m} \cap (\Gamma_{q,m}+t)): t \in \ss_{m,m}(q)}.
\end{align*}

\begin{theorem}\label{t16}\mbox{}
Assume that  $m_1=m_2=m$, and set
\begin{equation*}
c_1:=\frac{\log m}{\log q},\quad
c_2:=\frac{\log (m+1)}{\log q}.
\end{equation*}

\begin{enumerate}[\upshape (i)]
\item If $q\in [q_c,\infty)$, then $\hh_{m,m}(q)$ is dense in $[c_1,c_2]$.
\item If $q\in (m+1,q_c)$, then there exists a $\delta>0$ such that
\begin{equation*}
\hh_{m,m}(q)\cap(c_2-\delta,c_2)=\varnothing.
\end{equation*}
\end{enumerate}
\end{theorem}

In Section \ref{s2} we prove Theorem \ref{t14}, and we  illustrate the results by several examples.
In Section \ref{s3} we recall some notions and results on unique expansions in non-integer bases.
In Section \ref{s4} we collect several preliminary results that we need for the proof of the later theorems.
Then Theorems \ref{t11} and \ref{t12} are proved in Sections \ref{s5} and \ref{s6}, respectively.
Finally, Theorems \ref{t15} and \ref{t16} are proved in Section \ref{s7}.

\section{Proof of Theorem \ref{t14} and some examples}\label{s2}

We recall the statement of Theorem\ref{t14}.
Let $M$ be a positive integer $M$ and $\tau_0\cdots\tau_{\ell-1}\in\set{0,\ldots,M}^{\ell}$ a finite word  of length $\ell\ge 1$.
Defining the reflection of a word $w=c_1\cdots c_n$ by the formula
\begin{equation*}
\overline{w}=\overline{c_1\cdots c_n}
:=(M-c_1)\cdots(M-c_n),
\end{equation*}
we define an infinite sequence $(\tau_i)_{i=0}^{\infty}$ by the recursive formula
\begin{equation*}
\tau_{2^n\ell}\cdots\tau_{2^{n+1}\ell-1}:=\overline{\tau_0\cdots\tau_{2^n\ell-1}}\quad n=0,1,\ldots .
\end{equation*}

\begin{definition}\label{d21}
Given three words $\delta$, $\epsilon$ and $\zeta$, let
\begin{itemize}
\item  $\abs{\delta}$ denote the length of $\delta$;
\item  $N^\delta(\epsilon)$ denote the number of occurrences of $\delta$ in $\epsilon$, where $\abs{\delta}\le \abs{\epsilon}$;
\item  $\nu^\delta(\epsilon,\zeta)$ denote the number of those occurrences of $\delta$ in $\epsilon\zeta$ whose first and last letter belongs to $\epsilon$ and $\zeta$, respectively.
\end{itemize}
\end{definition}

\begin{t14}
Every finite word $\delta$ in $(\tau_i)_{i=0}^{\infty}$ has a density
\begin{equation*}
d(\delta):=\lim_{m\to\infty}\frac{N^{\delta}(\tau_0\cdots\tau_{m-1})}{m}.
\end{equation*}
Moreover, for every integer $n$ satisfying $4^n\ge \abs{\delta}-1$ we have
\begin{equation*}
d(\delta)=d(\overline{\delta})=\frac{P-N}{6\cdot 4^n\ell}
\end{equation*}
with
\begin{align*}
N&=N^\delta(\tau_0\cdots\tau_{4^n-1})
+N^{\overline{\delta}}(\tau_0\cdots\tau_{4^n-1})\intertext{and}
P&=N^\delta(\tau_0\cdots\tau_{4^{n+1}-1})
+N^{\overline{\delta}}(\tau_0\cdots\tau_{4^{n+1}-1}).
\end{align*}
\end{t14}

We need a lemma.

\begin{lemma}\label{l22}
Given three non-empty words $\delta$, $\epsilon$ and $\zeta$, we have the following relations:
\begin{align*}
&N^{\overline{\delta}}(\epsilon)=N^\delta(\overline{\epsilon}),\\
&N^\delta(\epsilon\zeta)=N^\delta(\epsilon)+N^\delta(\zeta)+\nu^\delta(\epsilon,\zeta),\\
&\nu^\delta(\epsilon,\zeta)\le \abs{\delta}-1 ,\\
&\abs{N^{\overline{\delta}}(\epsilon\overline{\epsilon})
-N^\delta(\epsilon\overline{\epsilon})}
\le \abs{\delta}-1.
\end{align*}
\end{lemma}

\begin{proof}
The first three properties follow from Definition \ref{d21}.
They imply the last inequality by choosing $\zeta=\overline{\epsilon}$.
\end{proof}

\begin{proof}[Proof of Theorem \ref{t14}]
We have to establish the following limit relations:
\begin{equation}\label{21}
\lim_{m\to\infty}\frac{N^{\delta}(\tau_0\cdots\tau_{m-1})}{m}
=\lim_{m\to\infty}\frac{N^{\overline{\delta}}(\tau_0\cdots\tau_{m-1})}{m}
=\frac{P-N}{6\cdot 4^n\ell}.
\end{equation}

We will use the following notation:
\begin{align*}
\epsilon&:=\tau_0\cdots\tau_{4^n\ell-1},\\
x_k&:=N^\delta(\tau_0\cdots\tau_{4^k\ell-1})
,\\
y_k&:=N^\delta(\overline{\tau_0\cdots\tau_{4^k\ell-1}}),\\
\alpha&:=\nu^\delta(\epsilon,\overline{\epsilon})+\nu^\delta(\overline{\epsilon},\overline{\epsilon})
+\nu^\delta(\overline{\epsilon},\epsilon),\\
\beta&:=\nu^{\overline{\delta}}(\epsilon,\overline{\epsilon})+\nu^{\overline{\delta}}(\overline{\epsilon},\overline{\epsilon})
+\nu^{\overline{\delta}}(\overline{\epsilon},\epsilon).
\end{align*}
We proceed in several steps.
\medskip

(i) We prove that
\begin{equation}\label{22}
x_{k+1}=2x_k+2y_k+\alpha
\qtq{and}
y_{k+1}=2x_k+2y_k+\beta
\end{equation}
for all $k\ge n$.
Writing $w_k:=\tau_0\cdots\tau_{4^k\ell-1}$ for brevity, it follows from the definition of $(\tau_i)$ that
\begin{equation*}
w_{k+1}=w_k \overline{w_kw_k} w_k,
\end{equation*}
and that $w_k$ starts and ends with the block $\epsilon$.
Since $\abs{\delta}\le\abs{\epsilon}+1$ by our assumption on $n$, distingishing seven different cases according to the position of the first letter of the occurrence of $\delta$ as a subword, it follows that
\begin{equation*}
N^{\delta}(w_{k+1})
=N^{\delta}(w_k)
+\nu^{\delta}(\epsilon,\overline{\epsilon})
+N^{\delta}(\overline{w_k})
+\nu^{\delta}(\overline{\epsilon},\overline{\epsilon})
+N^{\delta}(\overline{w_k})
+\nu^{\delta}(\overline{\epsilon},\epsilon)
+N^{\delta}(w_k).
\end{equation*}
This is equivalent to the first equality of \eqref{22}.
The second equality follows by changing $\delta$ to $\overline{\delta}$ in the above equality.
 \medskip

(ii) It follows from \eqref{22} that
\begin{equation*}
x_{k+1}+y_{k+1}+\frac{\alpha+\beta}{3}=4\left(x_k+y_k+\frac{\alpha+\beta}{3}\right),\qtq{for all}k\ge n.
\end{equation*}
For $k=n$ this may be written as
\begin{equation*}
P+\frac{\alpha+\beta}{3}=4\left(N+\frac{\alpha+\beta}{3}\right),
\end{equation*}
Hence $\alpha+\beta=P-4N$, and the following equality holds for all $k\ge n$:
\begin{align*}
x_{k+1}
&=2(x_k+y_k)+\alpha
=2\left(x_k+y_k+\frac{\alpha+\beta}{3}\right)+\frac{\alpha-2\beta}{3}\\
&=2\cdot 4^{k-n}\left(x_n+y_n+\frac{\alpha+\beta}{3}\right)+\frac{\alpha-2\beta}{3}\\
&=2\cdot 4^{k-n}\left(N+\frac{P-4N}{3}\right)+\frac{\alpha-2\beta}{3}\\
&=2\cdot 4^{k-n}\cdot\frac{P-N}{3}+\frac{\alpha-2\beta}{3}.
\end{align*}
Since $y_{k+1}=x_{k+1}+\beta-\alpha$ for all $k\ge n$ by \eqref{22}, it follows that
\begin{equation*}
\frac{x_{k+1}}{4^{k+1}}=\frac{P-N}{6\cdot 4^n}+\frac{\alpha-2\beta}{3\cdot 4^{k+1}}
\qtq{and}
\frac{y_{k+1}}{4^{k+1}}=\frac{P-N}{6\cdot 4^n}+\frac{\beta-2\alpha}{3\cdot 4^{k+1}}
\end{equation*}
for all $k\ge n$, and hence
\begin{equation}\label{23}
\lim_{k\to\infty}\frac{x_k}{4^k}
=\lim_{k\to\infty}\frac{y_k}{4^k}
=\frac{P-N}{6\cdot 4^n}.
\end{equation}
\medskip

(iii)
Fix an arbitrary integer $k\ge n$.
For any positive integer $p$, $\tau_0\cdots\tau_{p\cdot4^{k+1}\ell-1}$ is formed of $4p$ consecutive blocks of equal length, half of which are identical with $\tau_0\cdots\tau_{4^k\ell-1}$, and the other half with $\overline{\tau_0\cdots\tau_{4^k\ell-1}}$.
Therefore, by an argument similar to (i) we have
\begin{align*}
2p(x_k+y_k)\le N^{\delta}(\tau_0\cdots\tau_{p\cdot4^{k+1}\ell-1})
\le 2p(x_k+y_k)+(4p-1)(\abs{\delta}-1)
\intertext{and}
2p(x_k+y_k)\le N^{\overline{\delta}}(\tau_0\cdots\tau_{p\cdot4^{k+1}\ell-1})
\le 2p(x_k+y_k)+(4p-1)(\abs{\delta}-1).
\end{align*}
Now if $m\ge 4^{k+1}\ell$ is an integer, and  $p$ is the integer part of $m/(4^{k+1}\ell)$, then
\begin{equation}\label{24}
p\cdot 4^{k+1}\ell\le m<(p+1)\cdot 4^{k+1}\ell,
\end{equation}
and we infer from the above relations the following inequalities:
\begin{align*}
2p(x_k+y_k)\le N^{\delta}(\tau_0\cdots\tau_{m-1})
\le 2(p+1)(x_k+y_k)+(4p+3)(\abs{\delta}-1)
\intertext{and}
2p(x_k+y_k)\le N^{\overline{\delta}}(\tau_0\cdots\tau_{m-1})
\le 2(p+1)(x_k+y_k)+(4p+3)(\abs{\delta}-1).
\end{align*}
Combining them with \eqref{24} we obtain that
\begin{multline*}
\frac{2p}{4(p+1)}\left(\frac{x_k}{4^k\ell}+\frac{y_k}{4^k\ell}\right)
\le \frac{N^{\delta}(\tau_0\cdots\tau_{m-1})}{m}\\
\le \frac{2(p+1)}{4p}\left(\frac{x_k}{4^k\ell}+\frac{y_k}{4^k\ell}\right)
+\frac{(4p+3)(\abs{\delta}-1)}{4p\cdot 4^k\ell}
\end{multline*}
and
\begin{multline*}
\frac{2p}{4(p+1)}\left(\frac{x_k}{4^k\ell}+\frac{y_k}{4^k\ell}\right)
\le \frac{N^{\overline{\delta}}(\tau_0\cdots\tau_{m-1})}{m}\\
\le \frac{2(p+1)}{4p}\left(\frac{x_k}{4^k\ell}+\frac{y_k}{4^k\ell}\right)
+\frac{(4p+3)(\abs{\delta}-1)}{4p\cdot 4^k\ell}.
\end{multline*}

If $m\to\infty$, then $p\to\infty$, so that we infer from these inequalities the following relations:
\begin{multline*}
\frac{1}{2}\left(\frac{x_k}{4^k\ell}+\frac{y_k}{4^k\ell}\right)
\le\liminf_{m\to\infty}\frac{N^\delta(\tau_0\cdots\tau_{m-1})}{m}\\
\le\limsup_{m\to\infty}\frac{N^\delta(\tau_0\cdots\tau_{m-1})}{m}
\le \frac{1}{2}\left(\frac{x_k}{4^k\ell}+\frac{y_k}{4^k\ell}\right)
+\frac{\abs{\delta}-1}{4^k\ell}
\end{multline*}
and
\begin{multline*}
\frac{1}{2}\left(\frac{x_k}{4^k\ell}+\frac{y_k}{4^k\ell}\right)
\le\liminf_{m\to\infty}\frac{N^{\overline{\delta}}(\tau_0\cdots\tau_{m-1})}{m}\\
\le\limsup_{m\to\infty}\frac{N^{\overline{\delta}}(\tau_0\cdots\tau_{m-1})}{m}
\le \frac{1}{2}\left(\frac{x_k}{4^k\ell}+\frac{y_k}{4^k\ell}\right)
+\frac{\abs{\delta}-1}{4^k\ell}.
\end{multline*}
This is true for every $k\ge n$.
Letting $k\to\infty$ and using \eqref{23} we conclude \eqref{21}.
\end{proof}

\begin{examples}\label{e23}\mbox{}
In the following examples we consider the Thue--Morse sequence, i.e., $M=1$, $\ell=1$ and $\tau_0=0$.
\begin{enumerate}[\upshape (i)]
\item If $\delta=0$, then choosing $n=0$ we have $N=1$ and $P=4$, so that
\begin{equation*}
d^0=d^1=\frac{4-1}{6\cdot 1}=\frac{1}{2}.
\end{equation*}
\item If $\delta=01$, then choosing $n=1$ we have $N=2$ and $P=10$, so that
\begin{equation*}
d^{01}=d^{10}=\frac{10-2}{6\cdot 4}=\frac{1}{3}.
\end{equation*}
\item If $\delta=00$, then choosing $n=1$ we have $N=1$ and $P=5$, so that
\begin{equation*}
d^{00}=d^{11}=\frac{5-1}{6\cdot 4}=\frac{1}{6}.
\end{equation*}
\item If $\delta=000$, then choosing $n=1$ we have $N=P=0$, so that
\begin{equation*}
d^{000}=d^{111}=\frac{0-0}{6\cdot 4}=0.
\end{equation*}
\item If $\delta=001$, then choosing $n=1$ we have $N=1$ and $P=5$, so that
\begin{equation*}
d^{001}=d^{110}=\frac{5-1}{6\cdot 4}=\frac{1}{6}.
\end{equation*}
\item If $\delta=010$, then choosing $n=1$ we have $N=0$ and $P=4$, so that
\begin{equation*}
d^{010}=d^{101}=\frac{4-0}{6\cdot 4}=\frac{1}{6}.
\end{equation*}
\item If $\delta=011$, then choosing $n=1$ we have $N=1$ and $P=5$, so that
\begin{equation*}
d^{011}=d^{100}=\frac{5-1}{6\cdot 4}=\frac{1}{6}.
\end{equation*}
\item If $\delta=00101$, then choosing $n=1$ we have $N=0$ and $P=1$, so that
\begin{equation*}
d^{00101}=d^{11010}=\frac{1-0}{6\cdot 4}=\frac{1}{24}.
\end{equation*}
\end{enumerate}
\end{examples}

\begin{example}\label{e24}
Let $(\tau_i)$ be the Thue--Morse sequence.
We compute the frequencies of the digits $-1, 0, 1$ in the sequence $(\tau_i-\tau_{i-1})$.
Since $\tau_i\in\set{0,1}$ for every $i$, the frequency of the digit $-1$  in $(\tau_i-\tau_{i-1})$ is equal to the frequency of the words $10$ in $(\tau_i)$.
Applying Theorem \ref{t14} with $\delta=10$ and $n=1$ we obtain that
\begin{align*}
N&=N^{01}(\tau_0\cdots\tau_3)
+N^{10}(\tau_0\cdots\tau_3)\intertext{and}
P&=N^{01}(\tau_0\cdots\tau_{15})
+N^{10}(\tau_0\cdots\tau_{15}).
\end{align*}
Using the equalities
\begin{equation*}
\tau_0\cdots\tau_3=0110
\qtq{and}
\tau_0\cdots\tau_{15}=0110\ 1001\ 1001\ 0110
\end{equation*}
we obtain that $N=1+1=2$, $P=5+5=10$, and therefore
\begin{equation*}
d(01)=d(10)=\frac{P-N}{6\cdot 4}=\frac{10-2}{24}=\frac{1}{3}.
\end{equation*}
It follows that both digits $1$ and $-1$ have density $1/3$ in $(\tau_i-\tau_{i-1})$.
It follows that the third digit $0$ also has a density in $(\tau_i-\tau_{i-1})$, and it is also equal to $1/3$.
\end{example}

\section{Review of some results on unique expansions}\label{s3}

In this section we fix a positive integer $M$ and we consider expansions in a base $q>1$ over the alphabet $\set{0,1,\ldots,M}$. First we introduce some important notation and lemmas. We use the \emph{lexicographical order}. For any two sequences $(a_i)$ and $(b_i)\in \set{0,\ldots,M}^{\NN}$, we write $(a_i)\prec(b_i)$ if there exists a positive integer $n$ such that $a_i=b_i$ for $i=1,\ldots,n-1$ (this condition is void if $n=1$), and $a_n<b_n$.
Furthermore, we write $(a_i)\preceq(b_i)$ if $(a_i)\prec(b_i)$ or $(a_i)=(b_i)$.
Sometimes we write  $(b_i)\succ(a_i)$ instead of $(a_i)\prec(b_i)$, and $(b_i)\succeq(a_i)$ instead of $(a_i)\preceq(b_i)$.

We let $\overline{(a_i)}:=(M-a_i)$ denote the \emph{reflection} of a sequence $(a_i)$, and for a word $w=b_1\cdots b_n$ we write
\begin{align*}
&\overline{w}=(M-b_1)\cdots (M-b_n),\\
&w^{-}:=b_1\cdots b_{n-1}(b_n-1)\qtq{if}b_n>0,\\
&w^{+}:=b_1\cdots b_{n-1}(b_n+1)\qtq{if}b_n<M.
\end{align*}
We call a sequence $(a_i)$ \emph{infinite} if it does not have a last nonzero digit, and \emph{finite} otherwise.

We define the \emph{greedy expansion} of $x$ as the lexicographically largest expansion of $x$, and \emph{quasi-greedy expansion}  of $x$ as the lexicographically largest infinite expansion of $x$. We let $\uu_{M,q}$ denote the set of numbers $x \in \left[0,\frac{M}{q-1}\right] $ having a unique expansion in base $q$, and $\uu_{M,q}'\subset\set{0,\ldots,M}^{\NN}$ denote the set of the corresponding unique expansions.

We recall from \cite{KL,KL2,KL-2007,KLP-2011,K-2012,VKL,B}
that there exists a smallest base $q_{KL}$ in which $1$ has a unique expansion, and there exists a smallest base $q_1$ in which $1$ has a unique \emph{doubly infinite} expansion.
Here an infinite sequence $(c_i)$ is called doubly infinite if its \emph{reflection} $(M-c_i)$ is also infinite.
We have $1<q_1<q_{KL}<M+1$.
The base $q_1$ is given by the explicit formulas
\begin{equation}\label{31}
q_1=
\begin{cases}
(m+\sqrt{m^2+4m})/2&\text{if $M=2m-1$,}\\
m+1&\text{if $M=2m$,}
\end{cases}
\end{equation}
and the lexicographically largest (greedy) expansion  of $1$ in base $q_1$ is $\omega_10^{\infty}$ with
\begin{equation}\label{32}
\omega_1=
\begin{cases}
mm&\text{if $M=2m-1$,}\\
m+1&\text{if $M=2m$,}
\end{cases}
\quad m=1,2,\ldots .
\end{equation}

\begin{remark}\label{r31}
If we change $m$ to $m+1$ in the first formula of \eqref{31}, then we get the original expression in \cite{B}:
\begin{equation*}
\frac{m+1+\sqrt{(m+1)^2+4(m+1)}}{2}
=\frac{m+1+\sqrt{m^2+6m+5}}{2}.
\end{equation*}
\end{remark}

The base $q_{KL}$ is a transcendental number \cite{AC,KL2,DL}, and the unique expansion $(p_i)_{i=1}^{\infty}$ of $1$ in this base is given by the formulas
\begin{equation}\label{33}
p_i=
\begin{cases}
m-1+\tau_i&\text{if}\quad M=2m-1,\\
m+\tau_i-\tau_{i-1}&\text{if}\quad M=2m,
\end{cases}
\quad i=1,2,\ldots
\end{equation}
where
\begin{equation*}
(\tau_i)_{i=0}^{\infty}=0110\ 1001\ 1001\ 0110\ \cdots
\end{equation*}
denotes the Thue--Morse sequence, defined in the introduction.
We note that
\begin{equation}\label{34}
(p_i)\in
\begin{cases}
\set{m-1,m}^{\NN}&\text{if}\quad M=2m-1,\\
\set{m-1,m,m+1}^{\NN}&\text{if}\quad M=2m.
\end{cases}
\end{equation}

It was discovered in \cite[Corollary 15]{SN2} and generalized in \cite{B,VK,DLD} that the bases $q_1$ and $q_{KL}$ are critical for the size of the sets $\uu_{M,q}$ of  real numbers having a unique expansion in base $q$ over the alphabet $\set{0,1,\ldots,M}$, namely:
\begin{itemize}
\item $\uu_{M,q}$ has only the trivial elements $0$ and $\frac{M}{q-1}$ if $1<q\le q_1$;
\item $\uu_{M,q}$ is countably infinite if $q_1<q<q_{KL}$;
\item $\uu_{M,q}$ has the cardinality of the continuum if $q=q_{KL}$;
\item $\uu_{M,q}$ has a positive Hausdorff dimension if $q>q_{KL}$.
\end{itemize}

We will also use some other bases related with to the Thue--Morse sequence.
The number $1$ has a finite greedy expansion of the form $\omega_1\ 0^{\infty}$ in base $q_1$ where $\omega_1$ is a finite word having a last nonzero digit.
Starting with  $\omega_1$ we define by induction a sequence of words by the formula
\begin{equation*}
\omega_{k+1}:=(\omega_k\overline{\omega_k})^+,\quad k=1,2,\ldots .
\end{equation*}
Then $\omega_k0^{\infty}$ is the greedy expansion of $1$ in some base $q_k$, and $(q_k)_{k=1}^{\infty}$ is an increasing sequence converging to $q_{KL}$; see \cite{KL-2007,VKL} for details.
Let us note that the words $\omega_k$ correspond to the construction of the so-called  $q$-mirror sequences described in \cite[Section III-1]{A}.

If $M$ is even, then $q_1$ is an integer.
If $M$ is odd, then $q_1$ is not an integer, and we let $q_0$ denote   its integer part.
Then $1$ also has a finite greedy expansion of the form $\omega_0\ 0^{\infty}$ in base $q_0$ where $\omega_0$ is a finite word having a last nonzero digit, and
\begin{equation*}
\omega_1:=(\omega_0\overline{\omega_0})^+.
\end{equation*}

We recall from \cite[Corollary 7.1 and Remark 5.7]{KK}  that if $q_k<q\le q_{k+1}$ for some $k\ge 1$ and $(t_i)\in\uu_{M,q}'$, then  $(t_i)$ or its reflection $(M-t_i)$ is eventually periodic with one of the period blocks
\begin{equation*}
0, \omega_0\overline{\omega_0}, \omega_1\overline{\omega_1},\ldots, \omega_{k-1}\overline{\omega_{k-1}}
\end{equation*}
if $M$ is odd, and
\begin{equation*}
0,\omega_0^-, \omega_1\overline{\omega_1}, \ldots, \omega_{k-1}\overline{\omega_{k-1}}
\end{equation*}
if $M$ is even; see \cite[Lemma 4.12]{DLD}.

\section{Preliminary lemmas}\label{s4}

The following two lemmas reduce to \cite[Lemma 3.3 and Theorem 3.4]{WLDX} if $m_1=m_2=1$ and $q>2$.

\begin{lemma}\label{l41}
Assume that  $q-1>m_1\ge m_2$.
If $t \in \Gamma_{q,m_1}-\Gamma_{q,m_2}$, then
\begin{equation*}
\Gamma_{q,m_1}\cap(\Gamma_{q,m_2}+t)
=\bigcup_{(t_i)}\set{\sum_{i=1}^{\infty}\frac{c_i}{q^i}\ :\
c_i\in\Omega_{m_1}\cap(\Omega_{m_2}+t_i)\qtq{for all}i},
\end{equation*}
where we take the union over all the expansions $(t_i)$ of $t$.
\end{lemma}

\begin{proof}
If $x\in \Gamma_{q,m_1}\cap(\Gamma_{q,m_2}+t)$, then there exist a sequence $(c_i)\in\Omega_{m_1}^{\infty}$ satisfying $x=\sum_{i=1}^{\infty}\frac{c_i}{q^i}$ and a sequence $(d_i)\in\Omega_{m_2}^{\infty}$ satisfying $x-t=\sum_{i=1}^{\infty}\frac{d_i}{q^i}$.
Then setting $t_i:=c_i-d_i$ we get an expansion $(t_i)\in\Omega_{-m_2,m_1}$ of $t$ in base $q$, and $c_i\in \Omega_{m_1}\cap(\Omega_{m_2}+t_i)$ for all $i$ by the definition of $t_i$.

Conversely, if $x=\sum_{i=1}^{\infty}\frac{c_i}{q^i}$ for some sequence $(c_i)$ satisfying for every $i$ the condition $c_i\in \Omega_{m_1}\cap(\Omega_{m_2}+t_i)$ for some expansion $(t_i)\in\Omega_{-m_2,m_1}$ of $t$ in base $q$, then $x\in \Gamma_{q,m_1}$, and $\sum_{i=1}^{\infty}\frac{c_i-t_i}{q^i}$ is an expansion of $x-t$ in base $q$ over the alphabet $\Omega_{m_2}$, so that  $x\in \Gamma_{q,m_2}+t$.
\end{proof}

Now, following  \cite{WLDX}, we recall some results on generalized Moran sets \cite{H,HW,WLDX2}.
For each $j\in\NN$, let $H_j$ be a nonempty finite set of consecutive integers.
Set
\begin{equation*}
S^k:=\prod_{j=1}^kH_j\qtq{for all}k\in\NN,\qtq{and}S^*:=\cup_{k\ge1} S^k.
\end{equation*}
Let us also fix for each $j\in\NN$ and $i\in H_j$ a real number $r^j_i\in (0,1)$.

Given a nonempty compact set $J\subset \RR^d$ with $\overline{\mathrm{int}J}=J$,  assume that there exists a family $\set{J_{\sigma}\ :\ \sigma\in S^*}$
of nonempty compact subsets of $J$ having the following properties:

\begin{enumerate}[\upshape (a)]
\item $\overline{\mathrm{int}J_\sigma}=J_\sigma$ for every $\sigma\in S^*$;
\item for each $k\ge 1$ and $\sigma\in S^k$, the sets $J_{(\sigma,i)}$, $i\in H_{k+1}$ are non-overlapping subsets of $J_\sigma$, and their diameters  satisfy the following equalities:
\begin{equation*}
\abs{J_{(\sigma,i)}}=r_i^{k+1}\abs{J_\sigma}
\qtq{for all}
i\in H_{k+1}.
\end{equation*}
\end{enumerate}
Then
\begin{equation*}
E:=\bigcap_{k=1}^{\infty}\bigcup_{\sigma\in S^k}J_\sigma
\end{equation*}
is nonempty compact set, called a \emph{generalized Moran set}.

\begin{example}\label{e42}
If $t$ has a unique expansion $(t_j)$ over $\Omega_{-m_2,m_1}$, then
$\Gamma_{q,m_1} \cap (\Gamma_{q,m_2}+t)$ is a generalized Moran set.
Indeed, set $H_j:=\Omega_{m_1}\cap(\Omega_{m_2}+t_j)$ for all $j\in\NN$,  and $r^j_i:=1/q$ for all $j\in\NN$ and $i\in H_j$.
Furthermore, let us introduce the maps $f_i(x):=(x+i)/q$ for $i=0,\ldots,m_1$, and then the composed maps
\begin{equation*}
f_{\sigma}:=f_{\sigma_k}\circ\cdots\circ f_{\sigma_1}\qtq{for all}\sigma=(\sigma_1,\ldots,\sigma_k)\in S^k,\quad k\in\NN.
\end{equation*}
Finally, define $J:=[0,m_1/(q-1)]$, and $J_{\sigma}:=f_{\sigma}(J)$ for all $\sigma\in S^*$.
Then the conditions (a) and (b) are satisfied, so that
\begin{equation*}
\Gamma_{q,m_1} \cap (\Gamma_{q,m_2}+t)=\bigcap_{k=1}^{\infty}\bigcup_{\sigma\in S^k}J_\sigma
\end{equation*}
is a generalized Moran set.
\end{example}

We recall from the \cite{WLDX} the following result on the Hausdorff dimension $\dim_H E$ of a generalized Moran sets, proved in \cite{H,HW}:

\begin{theorem}\label{t43}
Consider a generalized Moran set, and assume the following three conditions:
\begin{enumerate}[\upshape (i)]
\item there exists a positive number $c$ such that for each $\sigma\in S^*$, $J_\sigma$ contains a ball of diameter $c\abs{J_\sigma}$;
\item $\inf_{j\in\NN}\min_{i\in H_j}r^j_i>0$;
\item $\lim_{k \to \infty}\sup_{(\sigma_1,\ldots,\sigma_k)\in S^k}\prod_{j=1}^kr^j_{\sigma_j}=0$.
\end{enumerate}
Then we may define a sequence $(s_k)$ of  real numbers by the formula
\begin{equation}\label{41}
\prod_{j=1}^k \left[\sum_{i\in H_j}(r_i^j)^{s_k}\right]=1,\quad k=1,2,\ldots,
\end{equation}
and the following equality holds:
\begin{equation*}
\dim_H E=\liminf_{k\to\infty}s_k.
\end{equation*}
\end{theorem}

As an application of Theorem \ref{t43}, we have the following result:

\begin{lemma}\label{l44}
Assume that  $q-1>m_1\ge m_2$,  let
$(t_j)\in\uu_{m_1,m_2}'(q)$, and set
\begin{equation}\label{42}
n_j=
\begin{cases}
m_2+1+t_j&\text{if\quad $t_j<0$,}\\
m_2+1&\text{if\quad $0\le t_j\le m_1-m_2$,}\\
m_1+1-t_j&\text{if\quad $t_j>m_1-m_2$.}
\end{cases}
\end{equation}
Then
\begin{equation*}
\dim_H(\Gamma_{q,m_1} \cap (\Gamma_{q,m_2}+t))=\liminf_{k\to\infty}\frac{\sum_{j=1}^k\log n_j}{k\log q}.
\end{equation*}
\end{lemma}

\begin{proof}
Consider the generalized Moran set of Example \ref{e42}.
We claim that the number of elements  $n_j$ of each set $H_j$ is given by the formula \eqref{42}.
Indeed, the proof of Lemma \ref{l41} shows that every $x\in \Gamma_{q,m_1}\cap(\Gamma_{q,m_2}+t)$ has a unique expansion $(x_j)$ in base $q$ over the alphabet $\Omega_{m_1}$, and $(x_j)$ is characterized by the conditions  $x_j\in\set{0,1,\ldots,m_1}\cap(\set{0,1,\ldots,m_2}+t_j)$ for all $j$.
They are equivalent to the inequalities
\begin{equation*}
0\le x_j\le m_1
\qtq{and}
0\le x_j-t_j\le m_2
\end{equation*}
or to
\begin{equation*}
\max\set{0,t_j}\le x_j\le \min\set{m_1,m_2+t_j},
\end{equation*}
so that
\begin{equation*}
n_j=\min\set{m_1,m_2+t_j}-\max\set{0,t_j}+1.
\end{equation*}
Since $m_1\ge m_2$ by our assumption, this is equivalent to \eqref{42}.

Applying Theorem \ref{t43} it suffices to
observe that the definition \eqref{41}
 of the exponents $s_k$ reduces in our case to
\begin{equation*}
\prod_{j=1}^k \left(n_jq^{-s_k}\right)=1,\end{equation*}
and this is equivalent to
\begin{equation*}
s_k=\frac{\sum_{j=1}^k\log n_j}{k\log q}.\qedhere
\end{equation*}
\end{proof}

\section{Proof of Theorem \ref{t11}}\label{s5}

Let us restate the theorem for convenience:

\begin{t11}
Let $q>1$.
Then:
\begin{enumerate}[\upshape (i)]
\item If $q-1\le m_2\le m_1$, then
\begin{equation*}
\dd_{m_1,m_2}(q)=\set{0,1}.
\end{equation*}
\item If $m_2<q-1\le m_1$, then
\begin{equation*}
\dd_{m_1,m_2}(q)=\set{0,\frac{\log (m_2+1)}{\log q}}.
\end{equation*}
\item If $m_2<m_1<q-1$, then
$\dd_{m_1,m_2}(q)$ contains the point $\frac{\log (m_2+1)}{\log q}$.\end{enumerate}
\end{t11}

We need a lemma.
First we observe that by \eqref{12},  \eqref{13} and \eqref{33} the Komornik--Loreti constant $q_{KL}$ of the alphabet $\set{-m_2,\ldots,m_1}$ is given by the formula
\begin{equation}\label{51}
(t_j)=
\begin{cases}
\left(\frac{m_1-m_2-1}{2}+\tau_j\right)_{j=1}^{\infty}
&\text{if $m_1+m_2$ is odd,}\\
\left(\frac{m_1-m_2}{2}+\tau_j-\tau_{j-1}\right)_{j=1}^{\infty}
&\text{if $m_1+m_2$ is even.}\\
\end{cases}
\end{equation}

\begin{lemma}\label{l51}
If $m_1>m_2$, then $m_2+1<q_{KL}<m_1+1$.
\end{lemma}

\begin{proof}
If $m_1+m_2$ is odd, say $m_1+m_2=2m-1$, then $m<q_{KL}<m+1$ because
\begin{equation*}
q_{KL}>q_1=\frac{m+\sqrt{m^2+4m}}{2}>m
\end{equation*}
and
\begin{equation*}
\sum_{j=1}^{\infty}\frac{m-1+\tau_j}{(m+1)^j}
<\sum_{j=1}^{\infty}\frac{m}{(m+1)^j}
=1
=\sum_{j=1}^{\infty}\frac{m-1+\tau_j}{q_{KL}^j}.
\end{equation*}
Since $m_2+1\le m\le m_1$,  the inequalities $m_2+1<q_{KL}<m_1+1$ follow.

If $m_1+m_2$ is even, say $m_1+m_2=2m$, then $m<q_{KL}<m+2$ because
\begin{equation*}
q_{KL}>q_1=m+1
\end{equation*}
and
\begin{equation*}
\sum_{j=1}^{\infty}\frac{m+\tau_j-\tau_{j-1}}{(m+2)^j}
<\sum_{j=1}^{\infty}\frac{m+1}{(m+2)^j}
=1
=\sum_{j=1}^{\infty}\frac{m+\tau_j-\tau_{j-1}}{q_{KL}^j}.
\end{equation*}
Since $m_1>m_2$, we have $m_2+1\le m\le m_1-1$,  the inequalities $m_2+1<q_{KL}<m_1+1$ follow again.

\end{proof}

\begin{proof}[Proof of Theorem \ref{t11}]
(i) If $1<q\le m_2+1$, then both $\Gamma_{q,m_2}$ and $\Gamma_{q,m_1}$  are intervals.
Therefore $\Gamma_{q,m_1} \cap (\Gamma_{q,m_2}+t)$ is also an interval, and hence it has either dimension zero or dimension one.
\medskip

(ii) If $m_2+1<q\le m_1+1$, then $\Gamma_{q,m_2}$ is a homogeneous Cantor set of dimension
\begin{equation*}
d:=\frac{\log (m_2+1)}{\log q}.
\end{equation*}
Therefore, similarly to the case (i),  $\Gamma_{q,m_1} \cap (\Gamma_{q,m_2}+t)$ has either dimension zero or $d$.
\medskip

(iii)
We are going to apply Lemma \ref{l44}.
Since $m_2<m_1<q-1$ by assumption,
we have $q>m_1+1>q_{KL}$ by Lemma \ref{l51}, and therefore the sequence $(t_j)$ defined by \eqref{51} belongs to $\uu_{m_1,m_2}'(q)$.

If  $m_1+m_2$ is odd, then
\begin{equation*}
t_j=\frac{m_1-m_2-1+2\tau_j}{2}
\end{equation*}
for all $j$.
Since
\begin{equation*}
0\le\frac{m_1-m_2+i}{2}\le m_1-m_2\qtq{for} i=-1,1,
\end{equation*}
we have $n_j=m_2+1$ for all $j$ in formula \eqref{42} of Lemma \ref{l44}, and hence $\frac{\log (m_2+1)}{\log q}$ belongs to $\dd_{m_1,m_2}(q)$.

If  $m_1+m_2$ is even, then
\begin{equation*}
t_j=\frac{m_1-m_2}{2}+\tau_j-\tau_{j-1}\end{equation*}
for all $j$.
Since
\begin{equation*}
0\le \frac{m_1-m_2}{2}+i\le m_1-m_2\qtq{for} i=-1,0,1,
\end{equation*}
we have $n_j=m_2+1$ for all $j$ in formula \eqref{42} of Lemma \ref{l44} again, and hence $\frac{\log (m_2+1)}{\log q}\in\dd_{m_1,m_2}(q)$.
\end{proof}

\section{Proof of Theorem \ref{t12}}\label{s6}

We restate the theorem again:

\begin{t12}
Assume that  $m_1=m_2=m$.
Then:
\begin{enumerate}[\upshape (i)]
\item If $q_{KL}<q<\infty$, then $\dd_{m,m}(q)$ contains an interval.
\item If $q=q_{KL}$,  then $\dd_{m,m}(q)$ is formed by the numbers
\begin{align*}
&0,\quad \frac{\log (m+1)}{\log q},
\quad \frac{\log [m^2(m+1)]}{3\log q},\intertext{and}
&\frac{\log m}{\log q}-\frac{\log \frac{m+1}{m}}{\log q}\sum_{i=1}^j\left(-\frac{1}{2}\right)^i
\qtq{for}1\le j<\infty.
\end{align*}
\item If  $q_k<q\le q_{k+1}$ for some $k\ge 1$, then $\dd_{m,m}(q)$ is formed by the numbers
\begin{equation*}
0,\quad \frac{\log (m+1)}{\log q},\qtq{and}
\frac{\log m}{\log q}
-\frac{\log \frac{m+1}{m}}{\log q}\sum_{i=1}^j\left(-\frac{1}{2}\right)^i
\qtq{for $1\le j<k$ if $k\ge 2$.}
\end{equation*}
More precisely, if $(m+t_i)$ ends with with $(\omega_j\overline{\omega_j})^{\infty}$ for some $j\ge1$, then
\begin{equation*}
\dim_H(\Gamma_{q,m} \cap (\Gamma_{q,m}+t))
=\frac{\log m}{\log q}
-\frac{\log \frac{m+1}{m}}{\log q}\sum_{i=1}^j\left(-\frac{1}{2}\right)^i.
\end{equation*}
\end{enumerate}
\end{t12}

We need some preliminary results.
Let us consider the more general case where $m_1+m_2=2m$ is even, and set $\mu:=(m_1-m_2)/2$ for brevity.
In view of the relation $\Omega_{-m_2,m_1}+m_2=\Omega_{2m}$ we introduce the following sequence over the alphabet $\Omega_{-m_2,m_1}=\set{-m_2,\ldots,m_1}$, related to the Thue--Morse sequence (see \eqref{33}):
\begin{equation}\label{61}
(\lambda_i):=\left(\mu+\tau_i-\tau_{i-1}\right)_{i=1}^{\infty}\in\set{\mu -1,\mu ,\mu +1}^{\NN}.
\end{equation}

Given an arbitrary word $t_1\cdots t_n\in\set{\mu -1,\mu,\mu+1}^n$ of length $n\ge 1$, we introduce the  fractions
\begin{align*}
&d_1^*(t_1\cdots t_n):=\frac{\#\set{1\le i\le n: t_i=\mu -1}}{n},\\
&d_2^*(t_1\cdots t_n):=\frac{\#\set{1\le i\le n: t_i=\mu }}{n},\\
&d_3^*(t_1\cdots t_n):=\frac{\#\set{1\le i\le n: t_i=\mu +1}}{n}.
\end{align*}
Furthermore, given an arbitrary sequence $(t_i)\in\set{\mu -1,\mu,\mu+1}^{\NN}$, we define the \emph{densities}
\begin{equation*}
d_j((t_i)):=\lim_{n\to\infty}d_j^*(t_1\cdots t_n),\quad j=1,2,3,
\end{equation*}
when these limits exist.

If $(t_i)=(\lambda_i)$, then by \cite[Lemmas 3.1, 3.2]{BD} we have the following formulas:

\begin{lemma}\label{l61}
Let $m_1+m_2$ be even and consider the sequence $(\lambda_i)$ is defined by \eqref{61}.
Then:
\begin{align}
&d_2^*(\lambda_1 \cdots \lambda_{2^n})=-\sum_{i=1}^n\left(\frac{-1}{2}\right)^i
\qtq{for all}n \in \NN,\tag{i}\\
&d_1((\lambda_i))
=d_2((\lambda_i))
=d_3((\lambda_i))
=\frac{1}{3}.\tag{ii}
\end{align}
\end{lemma}

Let us give a new proof for the  limit relations:

\begin{proof}[Proof of Lemma \ref{l61} (ii)]
It follows from \eqref{61} that for each $j=-1,0,1$, the density of $\mu +j$ in the sequence $(\lambda_i)$ is equal to the density of $j$ in the sequence $(\tau_i-\tau_{i-1})$.
Therefore the limit relations follow from Example \ref{e24}.
\end{proof}

Let us introduce the notation
\begin{equation}\label{62}
a_n=\mu \lambda_1\cdots\lambda_{2^n-1}\qtq{and}
b_n=\left(\mu -1\right)\lambda_1\cdots\lambda_{2^n-1},
\end{equation}
and their reflection $\overline{a_n},\overline{b_n}$ where each digit $c_i$ is replaced by $2\mu-c_i$.
The following result is proved in \cite[p. 2829]{DLD}:

\begin{lemma}\label{l62}
Let $m_1+m_2$ be even and consider the sequence $(\lambda_i)$ is defined by \eqref{61}.
If $q>q_{KL}$, then there exists an $n \in \NN$ such that $\uu_{m_1,m_2}'(q)$ contains a subshift of finite type over the alphabet $\set{a_n,b_n,\overline{a_n},\overline{b_n}}$ with the adjacency matrix
\begin{equation*}
A=
\begin{pmatrix}
0 & 1 & 1 & 0   \\
0 & 0 & 1 & 0  \\
1 & 0 & 0 & 1  \\
1 & 0 & 0& 0
\end{pmatrix}
.
\end{equation*}
\end{lemma}

We deduce from Lemma \ref{l62}  the following statement:

\begin{lemma}\label{l63}
Assume that $m_1+m_2$ is even, and consider the words $w_1:=b_n\overline{a_nb_n}a_n$ and $w_2:=\overline{a_n}a_n$ as in Lemma \ref{l62}.
Then:

\begin{enumerate}[\upshape (i)]
\item We have $d_2^*(w_1)<d_2^*(w_2)$.
\item For any  real number $d \in [d_2^*(w_1),d_2^*(w_2)]$ there exists a sequence $(t_i)\in\set{w_1,w_2}^{\NN}$ such that all three densities $d_1((t_i))$, $d_2((t_i))$ and $d_3((t_i))$ exist, and $d_2((t_i))=d$.
\end{enumerate}
\end{lemma}

\begin{proof}
(i)  It follows froms \eqref{62} and the definition of the reflection that
\begin{equation*}
d_2^*(\overline{b_n})=d_2^*(b_n)
<d_2^*(a_n)=d_2^*(\overline{a_n}).
\end{equation*}
This implies the required inequality
\begin{equation*}
d_2^*(b_n\overline{a_nb_n}a_n)<d_2^*(\overline{a_n}a_n).
\end{equation*}
\medskip

(ii) For any fixed real number $d \in [d_2^*(w_1),d_2^*(w_2)]$, there exists a sequence $(n_j)$ of natural numbers  such that the sequence
\begin{equation*}
(t_i):=w_1^{n_1}w_2^{n_2}w_1^{n_3}w_2^{n_4}\cdots
\end{equation*}
satisfies the equality $d_2((t_i))=d$.
By Lemma \ref{l62} we have $(t_i)\in \uu_{m_1,m_2}'(q)$.
It remains to show that the densities $d_1((t_i))$ and $d_3((t_i))$ exist, too.

Since
\begin{equation*}
d_{\ell}((t_i))
=\lim_{k\to\infty}\frac{4d_{\ell}^*(w_1)\sum_{j=1}^kn_{2j-1}+2d_{\ell}^*(w_2)\sum_{j=1}^kn_{2j}}{4\sum_{j=1}^kn_{2j-1}+2\sum_{j=1}^kn_{2j}}
\end{equation*}
for $\ell=1,2,3$, setting
\begin{equation*}
r_k:=\frac{4\sum_{j=1}^kn_{2j-1}}{4\sum_{j=1}^kn_{2j-1}+2\sum_{j=1}^kn_{2j}},\quad k=1,2,\ldots
\end{equation*}
we have
\begin{equation}\label{63}
d_{\ell}((t_i))
=d_{\ell}^*(w_2)+[d_{\ell}^*(w_1)-d_{\ell}^*(w_2)]\cdot\lim_{k\to\infty}r_k,\quad \ell=1,2,3.
\end{equation}
Since $d_2^*(w_1)\ne d_2^*(w_2)$, the existence of the density $d_2((t_i))$ implies the convergence of the sequence $(r_k)$, and then the densities $d_1((t_i))$ and $d_3((t_i))$ also exist by \eqref{63}.
\end{proof}

\begin{proof}[Proof of Theorem \ref{t12}]
(i) Assume that $m_1=m_2=m$; then $q_{KL}>m+1$.

Let $q>q_{KL}$.
By Lemmas \ref{l62} and \ref{l63}, for each real number $d \in [d_2^*(w_1),d_2^*(w_2)]$ there exists a $t\in \uu_{m_1,m_2}(q)$ whose unique expansion  $(t_i)\in\uu_{m_1,m_2}'(q)$ satisfies the equality $d_2((t_i))=d$, and the densities $d_1((t_i))$ and $d_3((t_i))$ exist, too.

Applying  Lemma \ref{l44} we obtain the equality
\begin{align*}
\dim_H(\Gamma_{q,m}\cap(\Gamma_{q,m}+t))
&=\frac{\log(m+1)}{\log q}{d_2}((t_i))+\frac{\log m}{\log q}({d_1}((t_i))+{d_3}((t_i)))\\
&=d\cdot\frac{\log(m+1)}{\log q}+(1-d)\frac{\log m}{\log q}\\
&=d\cdot\frac{\log\frac{m+1}{m}}{\log q}+\frac{\log m}{\log q}.
\end{align*}
Since $d$ was chosen arbitrarily in the interval $[d_2^*(w_1),d_2^*(w_2)]$, we conclude that  $\dd_{m,m}(q)$ contains the interval
\begin{equation*}
\left[\frac{\log\frac{m+1}{m}}{\log q}d_2^*(w_1)+\frac{\log m}{\log q},\frac{\log\frac{m+1}{m}}{\log q}d_2^*(w_2)+\frac{\log m}{\log q}\right].
\end{equation*}
\medskip

(iii)
If $1<q\le q_1$,  then $\uu_{m,m}'(q)=\set{(-m)^\infty,m^\infty}$.
If $k=1$ and $(t_i)\in\uu_{m,m}'(q)\setminus\set{(-m)^\infty,m^\infty}$, then $(t_i)$ ends with $0^\infty$.

Henceforth we assume that $k\ge2$.
Since $m_1=m_2$, we have $(\lambda_i)=(\tau_i-\tau_{i-1})$ by \eqref{61}, and therefore $\lambda_i\in\set{-1,0,1}$ for all $i$.
Furthermore, since $q_k<q\le q_{k+1}$, $(m+t_i)$ is eventually periodic with period block $\omega_j\overline{\omega_j}$ for $1\le j<k$.
Therefore, writing
\begin{equation*}
u_n:=\lambda_1 \cdots \lambda_{2^n}
\qtq{and}
\overline{u_n}:=(-\lambda_1) \cdots (-\lambda_{2^n})
\end{equation*}
for brevity, and applying Lemma \ref{l44} we have
\begin{align*}
&\dim_H(\Gamma_{q,m} \cap (\Gamma_{q,m}+t))\\
&=\frac{[\log (m+1)]d_2^*(u_j \overline{u_j})+(\log m)[d_3^*(u_j \overline{u_j})+d_1^*(u_j \overline{u_j})]}{\log q}\\
&=\frac{[\log (m+1)][-\sum_{i=1}^j\left(-\frac{1}{2}\right)^i]+(\log m)[1+\sum_{i=1}^j\left(-\frac{1}{2}\right)^i]}{\log q}\\
&=-\frac{\log \frac{m+1}{m}}{\log q}\sum_{i=1}^j\left(-\frac{1}{2}\right)^i+\frac{\log m}{\log q}.
\end{align*}
The second equality follows from Lemma \ref{l61} and from the equalities $d_2^*(u_j)=d_2^*(u_j\overline{u_j})=-\sum_{i=1}^j\left(-\frac{1}{2}\right)^i$.
\medskip

(ii)
Since $q=q_{KL}$, the sequence $(t_i):=(\lambda_i)$ is the unique expansion of some number $t$.
Applying Lemmas \ref{l44} and \ref{l61} (ii) we have
\begin{align*}
\dim_H(\Gamma_{q,m} \cap (\Gamma_{q,m}+t))&=\frac{\log(m+1)d_2((\lambda_i))+(\log m)(d_1((\lambda_i))+d_3((\lambda_i)))}{\log q}\\
&=\frac{\log (m+1)+2\log m}{3\log q}\\
&=\frac{\log [m^2(m+1)]}{3\log q}.
\end{align*}
The remaining part of (ii) follows from (iii) and from the structure of the periods of the elements of $\uu_{m,m}'(q)$, because  $q_k\rightarrow q_{KL}$ as $k\rightarrow\infty$.
\end{proof}

\section{Proof of Theorems \ref{t15} and \ref{t16}}\label{s7}

For convenience we restate the theorems to be proved in this section:

\begin{t15}\mbox{}
Assume that  $m_1=m_2=m$, and set
\begin{equation*}
c_1:=\frac{\log m}{\log q},\quad
c_2:=\frac{\log (m+1)}{\log q}.
\end{equation*}

\begin{enumerate}[\upshape (i)]
\item If $q\in [q_c,\infty)$, then $\dd_{m,m}(q)$ contains the interval $[c_1,c_2]$.
\item If $q\in (m+1,q_c)$, then there exists a $\delta>0$ such that
\begin{equation*}
\dd_{m,m}(q)\cap(c_2-\delta,c_2)=\varnothing.
\end{equation*}
\end{enumerate}
\end{t15}

To state the second theorem  we recall the definition of the sets
\begin{align*}
\ss_{m,m}(q)
&:=\set{t \in \uu_{m,m}(q): \Gamma_{q,m} \cap (\Gamma_{q,m}+t) \; \textrm{is a self-similar set}},\\
\hh_{m,m}(q)&:=\set{\dim_H(\Gamma_{q,m} \cap (\Gamma_{q,m}+t)): t \in \ss_{m,m}(q)}.
\end{align*}

\begin{t16}\mbox{}
Assume that  $m_1=m_2=m$, and set
\begin{equation*}
c_1:=\frac{\log m}{\log q},\quad
c_2:=\frac{\log (m+1)}{\log q}.
\end{equation*}

\begin{enumerate}[\upshape (i)]
\item If $q\in [q_c,\infty)$, then $\hh_{m,m}(q)$ is dense in $[c_1,c_2]$.
\item If $q\in (m+1,q_c)$, then there exists a $\delta>0$ such that
\begin{equation*}
\hh_{m,m}(q)\cap(c_2-\delta,c_2)=\varnothing.
\end{equation*}
\end{enumerate}
\end{t16}

First we recall from \cite{DarKat1993,DarKat1995,KL2,BK} a technical lemma:

\begin{lemma}\label{l71}
Let $q\in(1,M+1]$ and $(\alpha_i(q))$ be the quasi-greedy expansion of $1$ in base $q$.
If $(d_i)$ is an expansion of $x$, then  $x\in \uu_{M,q}$ if and only if
\begin{align*}
(d_{n+i})\prec &(\alpha_i(q))  \qtq{whenever}d_1 \cdots d_n\neq M^n\intertext{and}
\overline{(d_{n+i})}\prec &(\alpha_i(q))  \qtq{whenever}d_1 \cdots d_n\neq 0^n.
\end{align*}
\end{lemma}

We also recall from \cite{BK} the following characterization of the quasi-greedy expansion of 1.

\begin{lemma}\label{l72}
Let $(\alpha_i(q))$ be  the quasi-greedy expansion of $1$ in some base $q\in(1,M+1]$.
Then the map $q \rightarrow (\alpha_i(q))$ is a strictly increasing bijection from the interval  $[1,\infty)$ onto the set of all infinite sequences $(\alpha_i(q)) \in \set{0,\ldots,M}^{\NN}$ satisfying
\begin{equation*}
\alpha_{k+1}\alpha_{k+2} \cdots \preceq \alpha_1\alpha_2 \cdots\qtq{for all}
k\geq 0.
\end{equation*}
\end{lemma}

Since $\Omega_{-m,m}+m=\Omega_{2m}$, it is natural to call a sequence $(d_i)\in\Omega_{-m,m}^{\NN}$ to be the quasi-greedy expansion of $x$ in base $q$ over the alphabet $\Omega_{-m,m}$ if  $(d_i+m)\in \Omega_{2m}^{\NN}$ is the quasi-greedy expansion of $x+\frac{m}{q-1}$ in base $q$ over the alphabet $\Omega_{2m}$.

Then we can modify Lemmas \ref{l71} and \ref{l72} to obtain the following lemmas for the alphabet $\Omega_{-m,m}$.

\begin{lemma}\label{l73}
Let $(d_i)$ be an expansion of $x$ over the alphabet $\Omega_{-m,m}$ in some base $q\in(1,M+1]$, and let $(\alpha_i'(q))$ denote the quasi-greedy expansion of $\frac{q-(m+1)}{q-1}$ over the alphabet $\Omega_{-m,m}$.
Then $(d_i)$ is unique expansion if and only if the following two conditions are satisfied:
\begin{align*}
(d_{n+i})\prec &(\alpha_i'(q))  \qtq{whenever}d_1 \cdots d_n\neq m^n
\intertext{and}
(-d_{n+i})\prec &(\alpha_i'(q))  \qtq{whenever}d_1 \cdots d_n\neq (-m)^n.
\end{align*}
\end{lemma}

\begin{lemma}\label{l74}
Let $q \in (1,2m+1]$.
The map $q \mapsto (\alpha_i'(q))$ is a strictly increasing bijection of the interval  $(1,2m+1]$ onto the set of all infinite sequences $(\alpha_i'(q)) \in \set{-m,\ldots,m}^{\NN}$ satisfying
\begin{equation*}
\alpha_{k+1}'\alpha_{k+2}' \cdots \preceq \alpha_1'\alpha_2' \cdots \qtq{for all}k\geq 0,
\end{equation*}
where $(\alpha_i'(q))$ is the quasi-greedy expansion of $\frac{q-(m+1)}{q-1}$ over alphabet $\Omega_{-m,m}$. \end{lemma}

Now consider the number $q_c=(m+2+\sqrt{m(m+4)})/2$, introduced in \cite{DLD}.
It follows from \eqref{16} and Lemma \ref{l71} that $(\alpha_i(q_c))=(m+1)m^{\infty}$, i.e., $(m+1)m^{\infty}$
is the quasi-greedy expansion of $1$ in base $q_c$ over the alphabet $\Omega_{2m}$.
Hence the quasi-greedy expansion of $\frac{q_c-(m+1)}{q_c-1}$ in base $q_c$ over $\Omega_{-m,m}$ is the following sequence:
\begin{equation}\label{71}
(\alpha_i^{'}(q_c))=10^{\infty}.
\end{equation}

\begin{lemma}\label{l75}
If $q\in[q_c,\infty)$ then for any sequence of natural numbers $(n_i)$, the \,sequence $(1(-1))^{n_1}0^{n_2}(1(-1))^{n_3}0^{n_4}\cdots$ belongs to $\uu_{m,m}'(q)$.
\end{lemma}

\begin{proof}
We generalize the proof of \cite[Lemma 4.2]{BD} given for $m=1$.
Fix an arbitrary sequence $(n_i)$ of natural numbers.
Then
\begin{equation*}
(-1)0^\infty\prec (c_{n+i})\prec 10^\infty
\end{equation*}
for all $n\ge 0$, so that $(c_i)\in\uu_{m,m}'(q_c)$.
The case $q>q_c$ hence follows because
$\uu_{m,m}'(q_c)\subset\uu_{m,m}'(q)$ for all $q\ge q_c$
 by Lemmas \ref{l73} and \ref{l74}.
\end{proof}

\begin{proof}[Proof of Theorem \ref{t15}]
(i) For any fixed number $\lambda\in[0,1]$, we can choose a sequence $(a_j)$ of positive integers such that the density of the zero digits in the sequence
\begin{equation*}
(t_j):=(1(-1))^{a_1}0^{a_2}(1(-1))^{a_{3}}0^{a_{4}}\cdots
\end{equation*}
is equal to $\lambda$.
By Lemma \ref{l75} we have  $(t_j)\in \uu_{m,m}'(q)$.
Since $t_j\in\set{-1,0,1}$ for every $j$, we may apply Lemma \ref{l44} with
\begin{equation*}
n_j=
\begin{cases}
m+1&\text{if $t_j=0$,}\\
m&\text{if $t_j\ne 0$}
\end{cases}
\end{equation*}
to conclude that
\begin{equation*}
\dd_{m,m}(q)=\lambda c_2+(1-\lambda)c_1.
\end{equation*}
Since $\lambda\in[0,1]$ was arbitrary, this proves the relation $[c_1,c_2]\subset\dd_{m,m}(q)$.
\medskip

(ii) We are going to apply Lemma \ref{l44}.
Let us observe that
\begin{equation}\label{72}
n_j=m+1\qtq{if}t_j=0,\qtq{and}n_j\le m\qtq{if}t_j\ne 0.
\end{equation}

Fix $q\in(m+1,q_c)$ arbitrarily.
Then $\alpha'(q)\prec 10^{\infty}$ and $\alpha'(q)$ starts with $1$, hence there exist two  integers $k\ge 0$ and $\ell<0$ such that $\alpha'(q)$ starts with $10^k\ell$.
Now fix an arbitrary $t\in\uu_{m,m}(q)$, and let $(t_j)$ be its unique expansion.
We distinguish four cases.
If $(t_j)$ ends with $0^{\infty}$, then applying Lemma \ref{l44} and observing that $n_j=m+1$ for all $j$,
we obtain that
\begin{equation}\label{73}
\dim_H(\Gamma_{q,m} \cap (\Gamma_{q,m}+t))=c_2.
\end{equation}
If $(t_j)$ contains at most a finite number of zero digits, then $n_j\le m$ for all but finitely many indices, and therefore
\begin{equation}\label{74}
\dim_H(\Gamma_{q,m} \cap (\Gamma_{q,m}+t))\le c_1
\end{equation}
by Lemma \ref{l44}.
Otherwise, there exists an index $n$ such that $c_n=0$ and $c_{n+1}\ne 0$.
If $c_{n+1}>0$, then applying Lemma \ref{l73} and using the fact that $\alpha'(q)$ starts with $10^k\ell$, we obtain that
\begin{equation*}
c_{n+1}c_{n+2}\cdots
10^{k_1}(-1)0^{k_2}10^{k_3}(-1)0^{k_4}\cdots
\end{equation*}
with suitable nonnegative integers
$k_j\le k$.
Using \eqref{72} and applying Lemma \ref{l44} we conclude that
\begin{equation}\label{75}
\dim_H(\Gamma_{q,m} \cap (\Gamma_{q,m}+t))\le \frac{c_1+kc_2}{k+1}.
\end{equation}
Similarly, if $c_n=0$ and $c_{n+1}<0$, then
\begin{equation*}
c_{n+1}c_{n+2}\cdots
(-1)0^{k_1}10^{k_2}(-1)0^{k_3}10^{k_4}\cdots
\end{equation*}
with suitable nonnegative integers $k_j$, and applying Lemma \ref{l44} we get \eqref{75} again.
Setting
\begin{equation*}
\delta:=\frac{c_2-c_1}{k+1},
\end{equation*}
we conclude from \eqref{73}, \eqref{74} and \eqref{75} that
\begin{equation*}
\dd_{m,m}(q)\cap(c_2-\delta,c_2)=\varnothing.\qedhere
\end{equation*}
\end{proof}

For the proof of our last theorem we recall  from \cite[Theorem 3.2]{DLD} the following result:

\begin{lemma}\label{l76}
Let $q\in(m+1,\infty)$, $(t_i)$ be the unique expansion of $t$.
Then $t\in S_{m,m}(q)$ if and only if $(m-|t_i|)$ is strongly eventually periodic, i.e., $(m-|t_i|)=IJ^{\infty}$ with two suitable words of the same length and satisfying the relation $I\preceq J$.
\end{lemma}

\begin{proof}[Proof of Theorem \ref{t16}]
(i) Fix $d\in [c_1,c_2]$ arbitrarily, and consider a corresponding sequence $(a_j)$ in the proof of Theorem \ref{t16} (i).
For any fixed positive even integer $k$,
if we replace $(a_j)$ by the periodic sequence
\begin{equation*}
((1(-1))^{a_1}0^{a_2}\cdots(1(-1))^{a_{k-1}}0^{a_{k}})^{\infty},
\end{equation*}
then the corresponding numbers $t_k$ belong to $\ss_{m,m}(q)$ by Lemma \ref{l76}, and
\begin{equation*}
\dim_H(\Gamma_{q,m} \cap (\Gamma_{q,m}+t_k))\to
\dim_H(\Gamma_{q,m} \cap (\Gamma_{q,m}+t))
\qtq{as}k\to\infty
\end{equation*}
by Lemma \ref{l44}.
\medskip

(ii) Since  $\hh_{m,m}(q)\subset\dd_{m,m}(q)$, this follows from Theorem \ref{t16} (ii).
\end{proof}
\medskip

\emph{Acknowledgement.}
The authors thank Wenxia Li,  Zhiqiang Wang and the anonymous referees for their help and comments to improve the presentation of our results.

\end{document}